\documentclass[12pt]{amsart}
\usepackage{amssymb}
\usepackage{amsfonts}
 \usepackage{amsfonts,amssymb}

\usepackage{mathrsfs}
\let\mathcal\mathscr

\usepackage[all,ps,cmtip]{xy}

\def\llra{\hbox to 10mm{\rightarrowfill}}

\def\lllra{\hbox to 15mm{\rightarrowfill}}

\def\phi{{\varphi}}

\def\cI{\mathcal{I}}

\def\cF{\mathcal{F}}

\def\cO{\mathcal{O}}

\def\cP{\mathcal{P}}

\def\cJ{\mathcal{J}}

\let\tilde\widetilde

\DeclareMathOperator{\codim}{codim}
\DeclareMathOperator{\Pic}{Pic}

\newtheorem{lemm}{Lemma}[section]
\newtheorem{theo}[lemm]{Theorem}
\newtheorem{coro}[lemm]{Corollary}

\newtheorem{prop}[lemm]{Proposition}

\newtheorem*{conj*}{Conjecture}

\theoremstyle{definition}
\newtheorem{defn}[lemm]{Definition}

\newtheorem{rema}[lemm]{Remark}

\newtheorem{conj}[lemm]{Conjecture}

\newtheorem{qu}[lemm]{Question}

\theoremstyle{remark}
\newtheorem*{remark*}{Remark}
\newtheorem*{note*}{Note}

\begin{document}
\title[Cohomological rank functions and Syzygies]{Cohomological rank functions and Syzygies of abelian varieties}
\author{Zhi Jiang}
\address{Shanghai center for mathematical sciences, Xingjiangwan campus, Fudan University, Shanghai 200438, P. R. China}
\email{zhijiang@fudan.edu.cn}
 \thanks{The author is partially supported by NSFC grants No. 11871155 and No. 11731004.}
\maketitle
\section{Introduction}
 Given a polarized abelian variety $(A, L)$ of dimension $g$, it is well-known that $L^{2}$ is basepoint free and $L^3$ is very ample. However it is a quite subtle question when $L$ is basepoint free or very ample. We may
 assume that the polarization type of $L$ is $(\delta_1, \ldots, \delta_g)$ where $\delta_1|\delta_2|\cdots   |\delta_g$. We are interested in the projective geometry of $(A, L)$ when $\delta_1=1$ in this article. Such
 polarizations are called primitive polarizations.

 In \cite{DHS}, the authors studied primitive polarizations of type $(1, \ldots, 1, \delta_g)$ by degeneration method and proved that when  $(A, L)$ is generic in the corresponding moduli space, then
\begin{itemize}
 \item $L$ is basepoint free iff $d>g$;
 \item when $d>g+1$, the linear system $|L|$ induces a birational morphism from $A$ to its image;
 \item when $d>2^g$, $L$ is very ample.
 \end{itemize}

On the other hand, let's recall the following famous conjecture which generalizes Fujita's conjecture (see \cite[Conjecture 5.4]{Kol1}).
\begin{conj}\label{fujita} Let $X$ be a smooth projective variety, $x\in X$ a point, $L$ be a nef and big line bundle on $X$. Assume that $$(L^{\dim X})>(\dim X)^{\dim X},$$ and  for any positive dimensional irreducible
subvariety $x\in Z$, $$(L^{\dim Z}\cdot Z)\geq (\dim X)^{\dim Z},$$ then $K_X+L$ is basepoint free at $x$.
\end{conj}

Since $L$ is big and nef,  it is easy to see by vanishing theorem that $K_X+L$ is basepoint free at $x\in X$ is equivalent to the fact that  $H^1(X, \cI_x\otimes \cO_X(K_X+L))$.
Hence the basic strategy to attack this conjecture is the following:  construct a singular $\mathbb Q$-divisor $D$ numerically equivalent to $cL$ for some rational number $0<c<1$ such that $x$ is an isolated component of
the non-klt locus $\mathrm{Nklt}(X, D)$ of $(X, D)$. Recall that $\mathrm{Nklt}(X, D)$ is simply  the subscheme defined by the multiplier ideal $\cJ(X, D)$. Then one can conclude by Nadel vanishing.

There are  other positive properties we like to know about  ample divisors.  Let $L$ be an ample line bundle on a projective variety $X$. We say that $L$ satisfies property $(N_p)$ if the first $p$ steps of the minimal
graded free resolution of the section algebra $R_L:=\bigoplus_mH^0(X, L^m)$ over the polynomial ring $S_L=\mathrm{Sym}^{\cdot} H^0(X,L)$ are linear (see for instance \cite{Lar1}). For instance, $(N_0)$ means that $L$ is
projectively normal and $(N_1)$ is equivalent to say that the homogeneous ideal of $X$ in $\mathbb P(H^0(X, L)^*)$ is generated by quadrics. We will also regard basepoint-freeness as condition $(N_{-1})$.

When $(A, L)$ is a polarized abelian variety,  Lazarsfeld, Pareschi, and Popa \cite{LPP}  provided a very nice criterion to tell when $L$ satisfies property $(N_p)$. This method is similar to the general approach  for
Fujita's conjecture. More explicitly, let $o$ be the origin of $A$ and define \begin{eqnarray*}r(L):=\min\{c\mid there\; exists\; a\; \mathbb Q-\mathrm{effective}\;divisor\; D\equiv cL\\ \; such\; that\;
\cJ(D)=\cI_o\}\end{eqnarray*}
Then the main result of \cite{LPP} essentially states that if $r(L)<\frac{1}{p+2}$, then $L$ satisfies property $(N_p)$.

Ito \cite{Ito1} and Lozovanu  \cite{Loz} asked the following  conjecture on abelian varieties.
  \begin{conj}\label{main-problem}
Let $(A, L)$ be a polarized abelian variety of dimension $g$. For $p$ a positive integer, denote by $D=\frac{1}{p+2}L$ a $\mathbb Q$-ample divisor.  Assume that   $$(D^{\dim B}\cdot B)>(\dim B)^{\dim B}$$ for all abelian
subvariety of $B$ of $A$. Then $L$ satisfies property $(N_p)$.
\end{conj}
Note that the condition in Conjecute \ref{main-problem} is much weaker than Conjecture \ref{fujita} restricted on abelian varieties: we just need to bound the intersection numbers with abelian subvarieties and the bounds
are also much smaller.

 Caucci  applied cohomological rank functions defined in \cite{JP} to study the higher syzygies of ample line bundles on abelian varieties (\cite{C}). He considered the basepoint freeness threshold of $L$ defined in
 \cite{JP}: $$\beta(L):=\mathrm{min}\{t\in \mathbb Q\mid h^1_{\cI_o, L}(t)=0\},$$ where $h^1_{\cI_o, L}(t)$ is the cohomological rank function of the ideal sheaf of the origin $\cI_o$ associated to $L$. Then Caucci proved
 that if $\beta(L)<\frac{1}{p+2}$, then $L$ satisfies property $(N_p)$.
  He then asked a variant of Conjecture \ref{main-problem}.
  \begin{conj}\label{main-problem1}
Let $(A, L)$ be a polarized abelian variety of dimension $g$. For $p$ a positive integer, denote by $D=\frac{1}{p+2}L$ a $\mathbb Q$-ample divisor.  Assume that   $$(D^{\dim B}\cdot B)>(\dim B)^{\dim B}$$ for all abelian
subvariety of $B$ of $A$. Then $\beta(L)<\frac{1}{p+2}$.
\end{conj}

This conjecture was confirmed by Ito in dimension $\leq 3$ (see \cite{Ito}).
Define  $r'(L)$ to be  \begin{eqnarray*} \min \{c\mid there\; exists\; a\; \mathbb Q-\mathrm{effective}\; divisor\; D\equiv cL\; \\ such\; that\; o\; is \; an\; isolated \;
  component \; of\; \mathrm{Nklt}(A, D)\}.\end{eqnarray*}
It is clear that $r'(L)<r(L)$.
Ito's important observation is that $\beta(L)\leq r'(L)\leq r(L)$ (see \cite[Proposition 1.10]{Ito}).

  \begin{conj}\label{main-problem1}
Let $(A, L)$ be a polarized abelian variety of dimension $g$. For $p$ a positive integer, denote by $D=\frac{1}{p+2}L$ a $\mathbb Q$-ample divisor.  Assume that   $$(D^{\dim B}\cdot B)>(\dim B)^{\dim B}$$ for all abelian
subvariety of $B$ of $A$. Then $\beta(L)<\frac{1}{p+2}$.
\end{conj}

 The main results of this article are the following.
 \begin{theo}\label{result}Assume that $(A, L)$ is a polarized abelian variety of dimension $g$. Let $p\geq -1$ be an integer. Assume that $$(L^{\dim B}\cdot B)>(2(p+2)(\dim B))^{\dim B}$$ for any positive-dimensional
 abelian subvariety $B$ of $A$. Then
 $\beta(L)<\frac{1}{p+2}$. In particular, $L$ satisfies property $(N_p)$.
\end{theo}

For generic polarized abelian varieties with special polarizations, we can prove Conjecture \ref{main-problem1}.
\begin{theo} Let $(A, L)$ be a very general polarized abelian variety with polarization type $(1,\ldots, 1, \delta_g)$. For some integer $p\geq -1$, assume that $(L^g)>((p+2)g)^g$, then $\beta(L)<\frac{1}{p+2}$ and hence
$L$ satisfies property $(N_p)$.
\end{theo}
Here the condition that $(A, L)$ is very general is explicit. We just require the space of Hodge classes is of dimension $1$ in each degree.
See Theorem \ref{verygeneral} for the more precise version of this  result.
  \begin{rema}
  We recall some known results in this direction.

For $p=-1$ i.e. basepoint-freeness, the condition $(L^g)>(2g)^g$, which implies that $h^0(A, L)>2^g\frac{g^{g-1}}{(g-1)!}>2^gg$, is  much stronger than the condition in  Proposition 2 of \cite{DHS}. However the proof in
\cite{DHS} says nothing about the locus in the moduli space of $(A, L)$ when $|L| $ is basepoint free.

For $p=0$ i.e. projective normality, Iyer proved that if $(L^g)>2^g(g!)^2$ and $A$ is simple, then $L$ satisfies property $(N_0)$. The bound in Theorem \ref{result} is already better than Iyer's bound when $A$ is simple and
$g\geq 10$.

Lazarsfeld, Pareschi, and Popa proved that if $(L^g)>\frac{1}{2}(4(p+2)g)^g$ and $(A, L)$ is very general, then $L$ satisfies property $(N_p)$ for any $p\geq -1$ (\cite[Corollary B]{LPP}).

Ito (\cite{Ito1} and \cite{Ito}) proved Conjecture \ref{main-problem1} for abelian surfaces and abelian threefolds. He also proved that (\cite[Proposition 3.3]{Ito}) if $(L^{\dim B}\cdot B)>(g(p+2)(\dim B))^{\dim B}$ for
any positive-dimensional abelian subvariety $B$ of $A$, then
 $\beta(L)<\frac{1}{p+2}$.
 \end{rema}

\subsection*{Acknowledgements}
The author thanks Giuseppe Pareschi for many exchanges in this topic and thanks Olivier Debarre for pointing out the reference \cite{Iy}. The author is grateful to Atsushi Ito for pointing out several mistakes in a previous
version of this paper and many helpful comments.

  \section{General restults}

 \subsection{Kawamata and Helmke's work}
Let $A$ be an abelian variety and $D$ a $\mathbb Q$-ample divisor on $A$. Assume that there exists an effective $\mathbb Q$-divisor $D'\equiv cD$  with $0<c<1$ a rational number such that $(A, D')$ is not klt. Then we have
the following
\begin{lemm}  \label{ind1}
There exists $c_1 $ such that $0<c_1-c<<1$ and an effective $\mathbb Q$-divisor $D_1'\equiv c_1D$ such that $(A, D_1')$ is a log canonical pair satisfying:
\begin{itemize}
\item[1)] $\mathrm{Nklt}(A, D_1')$ is a normal subvariety $Z_1$ of $A$ containing the origin $o\in A$;
\item[2)] $Z_1$ is the minimal lc center through $o$ and is smooth at $o$.
\end{itemize}
\end{lemm}
Note that in this case, the multiplier ideal sheaf $\cJ(A, D_1')$ (see \cite[Chapter 9]{Lar2} for the definition) is the ideal sheaf $\cI_{Z_1}$ of $Z_1$.
\begin{proof}
We first take $$c_1'=lct(D'):=inf\{t>0\mid \mathcal J(t\cdot D') \;\mathrm{is} \;\mathrm{non-trivial} \}.$$ Then $c_1'\leq 1$ is a rational number and $(A, c_1'D)$ is log canonical. We now consider the set $CLC(A, c_1'D')$
of log canonical centers of $(A, c_1D')$ (see \cite[Definition 1.3]{Kaw}). Let $W$ be a minimal lc center of $(A, c_1'D')$ of minimal dimension. We know that \cite[Proposition 1.5]{Kaw} the intersection of two lc centers of
$(A, c_1'D')$ is the union of lc centers $(A, c_1'D')$. Then for any subvariety $Y\in CLC(A, c_1'D')$, either $W\subset Y$ or $W\cap Y=\emptyset$.

Since $D$ is ample, we can then take $D_2\in |kD|$ for some $k>0$ sufficiently divisible such that $D_2$ contains $W$, $D_2$ does not contain any other subvarieties in $CLC(A, c_1'D')$, and $D_2$ is smooth away from $W$. By
Tie-breaking, for some $0<\epsilon_1 , \epsilon_2<<1$, the pair $(A, (1-\epsilon_1)c_1'D'+\epsilon_2D_2)$ is log canonical and its lc centers are contained in $W$. By induction on dimensions, we may assume that  $W$ is the
unique lc center of $(A, (1-\epsilon_1)c_1'D+\epsilon_2D_2)$.

Finally, let $z\in W$ be a smooth point of $W$ and considering the translation morphism $t_z: A\rightarrow A$, $x\rightarrow x+z$. Let $D_1'=t_z^*((1-\epsilon_1)c_1'D'+\epsilon_2D_2)$. We conclude that the pair $(A, D_1')$
satisfies both $1)$ and $2)$, where $Z_1=W-z$ is a translation of $W$.
\end{proof}

We also recall  Helmke's induction \cite[Proposition 3.2 and Corollary 4.6]{Hel}.
\begin{prop} \label{ind2}
Let $D$ be an $\mathbb Q$-effective ample divisor on $A$ such that $m_o(D)>g$.
Let  $D_1= c_1D$ be a $\mathbb Q$-effective divisor for some $0<c_1<1$ \footnote{Actually $c_1=\min\{c\mid (A, cD)\; is\; log\; canonical\; at \;o\}$ is the  log-canonical threshold of $(A, D)$ at $o$} such that $(A, D_1)$
is log canonical at $o\in A$ and $Z_1$ is a minimal lc center of $(A, D_1)$through $o$ of dimension $d$. If $(D^{d}\cdot Z_1)\geq g^d\binom{g-1}{d-1}$, then there exists a $\mathbb Q$-effective divisor $D_2'\equiv \epsilon
D$  such that  $c_2:=c_1+\epsilon<1$ and for $D_2=D_1+D_2'$, we have $(A, D_2)$ is log canonical at $o$ with the minimal lc center $Z_2$ through $o$ properly contained in $Z_1$.

\end{prop}

The subadjunction formula (see \cite{FG}, \cite{Kaw2}, and \cite{Kaw3}) will  be repeated applied.
\begin{theo}\label{subadjunction}Assume that $X$ is a smooth projective  variety and $(X, D)$ is a log canonical pair. Assume that $Z$ is a minimal lc center of $(X, D)$.  There exists an effective $\mathbb Q$-divisor $D_Z$
on $Z$ such that $(Z, D_Z)$ is klt and
 $$(K_X+D)|_Z\sim_{\mathbb Q}K_Z+D_{Z}.$$
\end{theo}
\subsection{Intersection numbers}

Assume that $(D^g)>g^g$. Then we know that there exists $D_1'\sim_{\mathbb Q}cD$ with $0<c<1$ such that $(A, D_1')$ is not klt at $o\in A$. We then apply Lemma \ref{ind1}. For some $0<c_1<1$, there exists $D_1\equiv c_1L$
such that an irreducible normal subvariety $Z_1$ of dimension $d_1$ is the unique lc center of $(A, D_1)$ and $o\in Z_1$ is a smooth point.

In general it seems difficult to estimate $(D^{d_1}\cdot Z_1)$, except when $Z_1$ is of dimension $1$ or codimension $1$.

Let $H$ be an effective divisor on $A$, we know that there exists a quotient of abelian varieties with connected fibers $\varphi_{h}: A\rightarrow A_h$ with an ample divisor $H'$ on $A_h$ such that $H$ is equivalent to
$\varphi_h^*H'$. We denote by $K_h$ the kernel of $\varphi_h$. Then the numerical dimension of $H$ is equal to $\dim A_h$.
\begin{lemm}\label{intersection}
Let $D$ be a  $\mathbb Q$-ample divisor on $A$.
\begin{itemize}
\item[1)] For any effective divisor $H$ on $A$ of numerical dimension $k$, we have
$$(D^{g-1}\cdot H)\geq \big(k!(D^{g-k}\cdot K_h)\big)^{\frac{1}{k}}(D^g)^{\frac{k-1}{k}},$$
In particular, if $H$ is ample, we have $(D^{g-1}\cdot H)\geq (D^g)^{\frac{g-1}{g}}.$
\item[2)] Assume that $C$ is a curve generating $A$, then $$(D\cdot C)\geq g\sqrt[g]{\frac{(D^g)}{g!}}>(D^g)^{\frac{1}{g}}.$$
\end{itemize}
\end{lemm}
\begin{proof}
By Hodge type inequalities (see \cite[Corollary 1.6.3]{Lar1}), we have $$(D^{g-1}\cdot H)\geq (D^{g-k}\cdot H^k)^{\frac{1}{k}}(D^g)^{\frac{k-1}{k}}.$$  We also have $(D^{g-k}\cdot H^k)\geq (D^{g-k}\cdot K_h)\cdot
(H'^k)_{A_h}\geq k! (D^{g-k}\cdot K_h)$.

  The second statement is simply the main result of \cite{Deb}.
\end{proof}

This result leads us to ask the following question.
 \begin{qu}\label{intersectionquestion} Let $(A, L)$ be a polarized abelian variety of dimension $g$. Let $Z\subset A$ be an irreducible (geometrically) non-degenerate subvariety of dimension $2\leq d\leq g-2$. Then we have
 $(Z\cdot L^d)\geq (L^g)^{\frac{d}{g}}$.
 \end{qu}
If the polarized abelian variety $(A, L)$ is very general, then the answer is affirmative.
\begin{defn}\label{Hodge}
We say that a polarized abelian variety $(A, L)$ is Hodge theoretically very general if $\dim_{\mathbb Q} H^{k, k}(A,\mathbb Q)=1$, where $$H^{k, k}(A, \mathbb Q):=H^{k, k}(A, \mathbb C)\cap H^{2k}(A, \mathbb Q)$$ is the
vector space of codimension-$k$ Hodge classes.
\end{defn}
 \begin{rema}
 We know by Mattuck (see \cite[Theorem 17.4.1]{BL}) that if $(A, L)$ is very general in the moduli space of polarized abelian varieties with polarization type $(\delta_1,\ldots, \delta_g)$, then $\dim_{\mathbb Q} H^{k,
 k}(A,\mathbb Q)=1$ for $0\leq k\leq g$.
 \end{rema}
 \begin{lemm}\label{Hodgenumber1}
 Assume that $(A, L)$ is Hodge theoretically very general. Then $(L^k\cdot Z)>(L^{g})^{\frac{k}{g}}$ for any irreducible subvariety $Z$ of codimension $k$.
 \end{lemm}

 \begin{proof}
 We know that $\frac{1}{(g-k)!\delta_1\cdots\delta_{g-k}}[L^{g-k}\ ]$ is integral and is a minimal cohomology class. Since $\dim_{\mathbb Q} H^{g-k, g-k}(A,\mathbb Q)=1$, for any irreducible subvariety $Z$ of dimesion $k$,
 its cohomology class $[Z]$ is a integral multiple of $\frac{1}{(g-k)!\delta_1\cdots\delta_{g-k}}[L^{g-k}\ ]$.
 Thus
 \begin{eqnarray*} (L^k\cdot Z)&\geq &\frac{(L^g)}{(g-k)!\delta_1\cdots\delta_{g-k}}\\
 &=&\frac{g!\delta_{g-k+1}\cdots \delta_{g}}{(g-k)!.}
 \end{eqnarray*}
 It is easy to see that $\frac{g!}{(g-k)!}>(g!)^{\frac{k}{g}}$ and $\delta_{g-k+1}\cdots\delta_{g}>(\delta_1\cdots\delta_g)^{\frac{k}{g}}$. Hence $(L^k\cdot Z)>(L^{g})^{\frac{k}{g}}$.
 \end{proof}
By a similar computation, we have the following result for generic polarized abelian varieties.
 \begin{theo}\label{verygeneral}
 Fix $p\geq -1$ an integer. Assume that $(A, L)$ is a polarized abelian variety of dimension $g$ of polarization type $(\delta_1 ,\ldots, , \delta_g)$. Assume that
 \begin{itemize}
 \item[(1)] $(L^g)>((p+2)g)^g$;
 \item[(2)] $(p+2)g\geq (g-i)\delta_i$ for $1\leq i\leq g$;
 \item[(3)]$\dim_{\mathbb Q} H^{k, k}(A,\mathbb Q)=1$ for all $1\leq k\leq g$,\end{itemize}then $r'(L)<\frac{1}{p+2}$ and hence $L$ satisfies property $(N_p)$.
 \end{theo}
 \begin{proof} Let $D=\frac{1}{p+2}L$.
Since $(D^g)>(g)^g$, we may assume that $D$ is an effective $\mathbb Q$-divisor with $m_o(D)>g$. Then there exists $D_1=c_1D$ for some $0<c_1<1$ such that
$(A, D_1) $ is log canonical at $o$ and the minimal lc center through $o$ is an normal subvariety $Z_1$ of dimension $k\geq 0$. If $k=0$, we know that $r'(L)=\frac{1}{p+2}r'(D)<\frac{1}{p+2}$ and we conclude by
\cite[Proposition 1.10]{Ito}.

If $k>0$,  by assumption $(2)$ and $(3)$,
we have
\begin{eqnarray*}(D^k\cdot Z_1)&\geq &\frac{1}{(p+2)^k}\frac{(L^g)}{(g-k)!\delta_1\cdots \delta_{g-k}}\geq \frac{1}{(p+2)^{k}}\frac{g!\delta_1\cdots\delta_g}{(g-k)!\delta_1\cdots\delta_{g-k}}\\
&\geq & (p+2)^{g-k}\frac{g^g}{(g-k)!\delta_1\cdots\delta_{g-k}}\geq g^k \frac{((p+2)g)^{g-k}}{(g-k)!\delta_1\cdots\delta_{g-k}}\\
&\geq &g^k\frac{(g-1)\cdots k}{(g-k)!}=g^k\binom{g-1}{k-1}.
\end{eqnarray*}
 Then by Proposition \ref{ind2}, there exists $D_2\equiv c_2D$ with $c_1<c_2<1$ such that $(A, D_2)$ is lc at $o$ and the minimal lc center through $o$ is $Z_2\subsetneqq Z_1$. We continue this process to yield that
 $r'(D)<1$ and conclude by \cite[Proposition 1.10]{Ito}.
 \end{proof}

\begin{coro} Assume that $(A, L)$ is a very general polarized abeian variety of type $(1,\ldots,1, \delta_g)$, then if $(L^g)>((p+2)g)^g$, $L$ satisfies property $(N_p)$.
\end{coro}

 \subsection{Generic vanishing and Cohomological rank functions}
 For an abelian variety $A$ of dimension $g$, we denote by $\Pic^0(A)$ the dual abelian variety of $A$. For any coherent sheaf $\cF$ on $A$, we denote by $$V^i(\cF):=\{P\in \Pic^0(A)\mid h^i(A, \cF\otimes P)\neq 0\}$$ the
 $i$-th cohomological support loci of $\cF$. We say that $\cF$ is IT$^0$ if $V^i(\cF)=\emptyset$ for $i\geq 1$. We say that $\cF$ is M-regular (resp. GV) if $\codim_{\Pic^0(A)}V^i(\cF)>i$ (resp.
 $\codim_{\Pic^0(A)}V^i(\cF)\geq i$) for $i\geq 1$. Another way to define M-regular or GV sheaves is through the Fourier-Mukai functor. Let $\cP$ be the normalized Poincar\'e line bundle on $A\times \Pic^0(A)$. We denote by
 $$\Phi_{\cP}:  \mathbf{D}^b(A)\rightarrow \mathbf{D}^b(A)$$  the Fourier-Mukai functor induced by $\cP$ between the derived categories of bounded complex of coherent sheaves on $A$ and $\Pic^0(A)$. Then  we know that a
 coherent sheaf $\cF$ is M-regular (resp. GV) if $\codim_{\Pic^0(A)}\mathrm{Supp}R^i\Phi_{\cP}(\cF)>i$ (resp. $\codim_{\Pic^0(A)}\mathrm{Supp}R^i\Phi_{\cP}(\cF)\geq i$) for $i\geq 1$ (see \cite{PP}).

 Let $a: X\rightarrow A$ be a morphism from a smooth projective variety to $A$ and $a$ is generically finite over its image. Then by \cite{GL}, we know that $a_*\omega_X$ is GV. Moreover, if $a(X)$ is not fibred by abelian
 subvarieties of $A$, then $a_*\omega_X$ is M-regular (see for instance \cite[Lemma 2.1]{JLT}).

We recall the definition of cohomological rank functions over an abelian variety $A$ of dimension $g$. Let $\cF$ be a coherent sheaf on $A$ and let $L$ be an ample line bundle on $A$, for a rational number $t\in \mathbb Q$,
$$h_{\cF, L}^i(t):=\frac{1}{M^{2g}}h^i(A,\pi_M^*\cF\otimes L^{M^2t}\otimes Q),$$ where $M$ is a sufficiently divisible integer such that $M^2t\in \mathbb Z$, $\pi_M: A\rightarrow A$ is the multiplication-by-$M$ map, and
$Q\in\Pic^0(A)$ is  general. It is easy to check that $h^i_{\cF, L}: \mathbb Q\rightarrow \mathbb Q$ is a well-defined map. In \cite{JP}, it has been proved that $h^i_{\cF, L}$ can be extended to a continuous function from
$\mathbb R$ to $\mathbb R$.

If $D\equiv cL$ for some rational number $c>0$, we can similarly define $h^i_{\cF, D}(t):=h^i_{\cF, L}(ct)$ for any $t\in \mathbb Q$.  For any ample $\mathbb Q$-divisor $D$, we define $$\beta(D):=\mathrm{min}\{t\in \mathbb
Q\mid h^1_{\cI_o, D}(t)=0\},$$ where $\cI_0\subset \cO_A$ is the ideal sheaf of the origin $o$ of $A$. It is clear that if an ample $\mathbb Q$-divisor $D$ is numerically equivalent to $cL$, then $\beta(D)=c\beta(L)$.

 We will use the following facts:
  \begin{itemize}\label{facts}
 \item[(1)] (\cite[Theorem 1.1]{C}) if $\beta(L)<\frac{1}{p+2}$, $L$ satisfies property   $(N_p)$;
 \item[(2)] (\cite[Proposition 3.1 and Theorem 3.2]{PP3})Assume that $\cF$ is a GV sheaf on $A$ and $V$ is a locally free IT$^0$ (resp. GV) sheaf on $A$, then $\cF\otimes V$ is still IT$^0$ (resp. GV);
  \item[(3)](\cite[Theorem 5.2]{JP}) Let $\cF$ be a GV sheaf on $A$, then $h^i_{\cF, L}(t)=0$ for $t>0$;
 \item[(4)] Given a short exact sequence of coherent sheaves on $A$
 $$0\rightarrow \cF_1\rightarrow \cF_2\rightarrow \cF_3\rightarrow 0,$$ then we have the long exact sequence $$\cdots h^0_{\cF_3, L}(t)\rightarrow h^1_{\cF_1, L}(t)\rightarrow h^1_{\cF_2, L}(t)\rightarrow h^1_{\cF_3,
 L}(t)\rightarrow \cdots.$$

\end{itemize}

 Given an irreducible subvariety $Z$ of $A$ containing the origin $o$ as a smooth point, we will denote by $\cI_{o, Z}$ the ideal sheaf of $o$ in $Z$. Thus we have $$0\rightarrow \cI_Z\rightarrow \cI_o\rightarrow \cI_{o,
 Z}\rightarrow 0.$$

 Combining $(4)$ above and Nadel vanishing, we have the following criterion.
 \begin{lemm}\label{criterion}
 Let $D$ be a $\mathbb Q$-ample divisor. Assume that  there exists $D'\equiv cD$ for some rational number $0<c<1$ such that $\cJ(D')=\cI_Z$ for some irreducible normal subvariety $Z$ containing $o$ as a smooth point, and
 $h^1_{\cI_{o, Z}, D}(1-\epsilon)=0$ for some $0<\epsilon<1-c$, then $\beta(D)<1$.
 \end{lemm}

\begin{coro}\label{subabelian} Assume that Conjecture \ref{main-problem1} holds in dimension $\leq g-1$.

Let $(A, L)$ be a polarized abelian variety of dimension $g$ and $p\geq -1$ an integer. Let $D=\frac{1}{p+2}L$. Assume that $(D^{\dim B}\cdot B)>(\dim B)^{\dim B}$ for all abelian subvarieties $B$ of $A$.
  If there exists an effective divisor $D'\equiv cD$ for some rational number $0<c<1$ such that $\cJ(D')=\cI_B$, where $B$ is an abelian subvariety of $A$, then $\beta(L)<\frac{1}{p+2}$.
\end{coro}
\begin{proof}
we just need to apply Lemma \ref{criterion} and the hypothesis that Conjecture \ref{main-problem1} holds in dimension $\leq g-1$.
\end{proof}

The above lemma is  simple but is quite crucial.

\begin{lemm}\label{ideal-GV}
Let $Z$ be a subvariety of $A$ containing $o$ as  a smooth point and let $\rho: Z'\rightarrow Z$ be a desingularization of $Z$. Assume that $Z'$ is of general type. Then for any $Q\in \Pic^0(Z')$, $\rho_*(\omega_{Z'}\otimes
Q)$ is M-regular and $\rho_*(\omega_{Z'}\otimes Q)\otimes \cI_{o, Z}$ is GV.
\end{lemm}
\begin{proof}
It is well-known that $\rho_*(\omega_{Z'}\otimes Q)$ is M-regular when $Z'$ is of general type (see for instance \cite[Lemma 2.1]{JLT}). Note that $\rho_*(\omega_{Z'}\otimes Q)$ is locally free of rank one around $o$, we
have $$o\rightarrow \rho_*(\omega_{Z'}\otimes Q)\otimes \cI_{o, Z}\rightarrow  \rho_*(\omega_{Z'}\otimes Q)\rightarrow \mathbb C_o\rightarrow 0.$$
Hence $V^j(\rho_*(\omega_{Z'}\otimes Q)\otimes \cI_{o, Z})=V^j(\rho_*(\omega_{Z'}\otimes Q))$ has codimension at least $j+2$ in $\Pic^0(A)$ for $j\geq 2$.
Since $\rho_*(\omega_{Z'}\otimes Q)$ is M-regular, it is continuous globally generated by \cite[Proposition 2.13]{PP}. Thus for $P\in \Pic^0(A)$ general, $H^0(\rho_*(\omega_{Z'}\otimes Q)\otimes P)\rightarrow \mathbb C_o $
is surjective. Hence $V^1(\rho_*(\omega_{Z'}\otimes Q)\otimes \cI_{o, Z})$ is a proper subset of $\Pic^0(A)$. Thus $\rho_*(\omega_{Z'}\otimes Q)\otimes \cI_{o, Z}$ is GV.
\end{proof}

\begin{rema}\label{ideal1}Assume that $\cF$ is a M-regular rank $1$ sheaf supported on $Z$, which is locally free around $o$,  then the same argument also shows that $\cF\otimes \cI_{o, Z}$ is GV.
\end{rema} 

The following result deals the cases when the minimal lc center is a divisor or is a divisor in an abelian subvariety of $A$.

  \begin{prop}\label{tech}
Assume that Conjecture \ref{main-problem1} holds in dimension $\leq g-1$.

  Let $(A, L)$ be a polarized abelian variety of dimension $g$ and $p\geq -1$ an integer. Let $D=\frac{1}{p+2}L$. Assume that $(D^{\dim B}\cdot B)>(\dim B)^{\dim B}$ for all abelian subvarieties $B$ of $A$ and there exists
  an effective divisor $D'\equiv cD$ for some rational number $0<c<1$ such that $\cJ(D')=\cI_Z$. Then in the following cases, we have $\beta(L)<\frac{1}{p+2}$:
  \begin{itemize}
  \item[(1)] $Z$ is a divisor on $A$;
  \item[(2)] $Z$ generates an abelian subvariety $B$, and there exists an integral divisor $H_B$ on $B$ such that $D'|_B-H_B$ is a nef  $\mathbb Q$-divisor on $B$ and $H_B|_Z\sim K_Z$ (thus $Z$ is Gorenstein).
  \end{itemize}
 \end{prop}

\subsection{The proof of Proposition \ref{tech} (1)}
We first assume that $Z$ is an ample divisor. After translation, we may assume that $o\in Z$ is a smooth point  and it suffices to show that $h^1_{\cI_{o, Z}, D}(1-\epsilon)=0$ for some $0<\epsilon<<1$ by Lemma
\ref{criterion}.

By adjunction, we know that $\cO_Z(Z)=\omega_Z$ is the dualizing sheaf on $Z$. From $$0\rightarrow \cO_A\rightarrow \cO_A(Z)\rightarrow \cO_Z(Z)\rightarrow 0,$$ we see easily that $\cO_Z(Z)$ is M-regular. Thus
$\cO_Z(Z)\otimes \cI_{o, Z}$ is GV by Remark \ref{ideal1}. Hence  $h^1_{\cO_Z(Z)\otimes \cI_{o, Z}, D}(\delta)=0$ for $0<\delta<<1$ by (\ref{facts}) (4).

We then take  $M>0$ sufficiently large and divisible such that
$M^2D'$ is an integral divisor. Since $D'=Z+Z'$  where $Z'$ is an effective $\mathbb Q$-divisor. We see that $M^2D'-M^2Z$ is an effective divisor on $A$ and hence $\cO_A(M^2D'-M^2Z)$ is GV. Let $\pi_: A\rightarrow A $ be
the multiplication-by-$M$ map.

By Subsection \ref{facts} (2), we know that $\pi_M^*(\cO_Z(Z)\otimes \cI_{o, Z})\otimes \cO_A(M^2D'-M^2Z)$ is GV. Hence $\pi_M^*(\cI_{o, Z})\otimes \cO_A(M^2D') $ is GV. Thus $h^1_{\cI_{o, Z}, D}(t)=0$ for $t>c$ by
Subsection \ref{facts} (3).

If $Z$ is not ample, we can again assume that $Z$ is fibred by abelian subvariety $K$ and we consider
\begin{eqnarray*}
\xymatrix{
Z\ar[d]^{h}  \ar@{^{(}->}[r] & A\ar[d] \\
Z_K \ar@{^{(}->}[r] & A/K.}
\end{eqnarray*}
We may assume that $o_K:=h(o)$ is a smooth point of $Z_K$ and let $K_o:=h^{-1}(o)$. Consider
$$0\rightarrow \cI_{K_o, Z}\rightarrow \cI_{o, Z}\rightarrow  \cI_{o, K_o}\rightarrow 0,$$
 hence as before, it suffices to show that $h^1_{\cI_{K_o, Z}, D}(1-\epsilon)=0$ and $h^1_{\cI_{o, K_o}, D}(1-\epsilon)=0$ for $0<\epsilon<<1$.

 The latter is  a consequence of the assumption that Conjecture \ref{main-problem1} holds in dimension $\leq g-1$.

 We note that $\cI_{K_o, Z}=h^*\cI_{o_K, Z_K}$ and $$\cO_Z(Z)\otimes \cI_{K_o, Z}=h^*(\omega_{Z_K}\otimes \cI_{o_K, Z_K})$$ is again GV. Then by exactly the same argument as the general type case, we conclude the proof.

\subsection{The proof of Proposition \ref{tech} (2)}
The proof is quite similar to the previous case.
After translation, we may assume that $o\in Z$ is a smooth point. By Theorem \ref{subadjunction}, $D'|_Z\sim_{\mathbb Q}K_Z+D_Z$ such that $(Z, D_Z)$ is klt. Moreover, by assumption, $Z$ is Gorenstein. Then $Z$ has
canonical singularities. For a resolution $\rho: Z'\rightarrow Z$ we have $\rho_*K_{Z'}=K_Z$. Thus by Lemma \ref{ideal-GV}, we know that $\cI_{o, Z}\otimes H_B$ is GV.

Take $D''=D'+\epsilon D$ such that $c<c'':=c+\epsilon<1$. Then $D''-H_B$ is an ample $\mathbb Q$-divisor on $B$.

Let $M$ be an integer sufficiently large and divisible so that $M^2D''$ is an integral divisor and let $\pi_M: A\rightarrow A$ be the multiplication-by-$M$ map. We have the commutative diagram
\begin{eqnarray*}
\xymatrix{
Z^{(M)}\ar[d]^{\mu_M}\ar@{^{(}->}[r] & B\ar[d]^{\pi_M}\ar@{^{(}->}[r] & A\ar[d]^{\pi_M}\\
Z\ar@{^{(}->}[r] & B\ar@{^{(}->}[r] & A,}
\end{eqnarray*}
where $Z^{(M)}:=\pi_M^{-1}(Z)$.

We take a general  $s\in H^0(B, \cO_B(M^2D''-M^2H_B))$ such that $o$ is not contained in the corresponding  divisor of $s$. Let $W$ be the zero locus of $s|_{Z^{(M)}}$. We have the following two exact sequences:
$$0\rightarrow \cO_{Z^{(M)}}(M^2H_B)\rightarrow  \cO_{Z^{(M)}}(M^2D'')\rightarrow \cO_{W}(M^2D'')\rightarrow 0$$
and
$$ 0\rightarrow\mu_M^{*}(\cI_{o, Z}(H_B))\rightarrow \pi_M^{*}(\cI_{o, Z})(M^2D'')\rightarrow\cO_{W}(M^2D'')\rightarrow 0.$$

In the first exact sequence,  we know that $M^2H_B|_{Z^{(M)}}=K_{Z^{(M)}}$, hence $ \cO_{Z^{(M)}}(M^2H_B)$ is GV. Moreover, since $\cJ(D')=\cI_Z$, we have $\cJ(\pi_M^*D')=\cI_{Z^{(M)}}$, we then see from Nadel vanishing
that $\cO_{Z^{(M)}}(M^2D'')$ is IT$^0$. Thus $\cO_{W}(M^2D'')$ is M-regular.

In the second exact sequence, we already know that $\mu_M^{*}(\cI_{o, Z}(H_B))$ is GV, thus $\pi_M^{*}(\cI_{o, Z})(M^2D'')$ is also GV and we then conclude that $\beta(L)<\frac{1}{p+2}$.
\section{The main result}

\begin{theo}\label{main}
Let $(A, L)$ be a polarized abelian variety of dimension $g$ and $p\geq -1$ be an integer. Let $D=\frac{1}{2(p+2)}L$ be a $\mathbb Q$-Cartier divisor. Assume that $$(D^{\dim B}\cdot B)>(\dim B)^{\dim B},$$ then $L$
satisfies property $(N_p)$.
\end{theo}
\begin{proof}Since $(D^g)>g^g$, by Lemma \ref{ind1}, we may take $D_1\equiv c_1D$ for some $0<c_1<1$ such that an irreducible normal variety $Z$ is the unique lc center of $(A, D_1)$ containing $o$ as a smooth point . Hence
$\cJ(A, D_1)=\cI_{Z}$. By Lemma \ref{criterion}, it suffices to show that $h^1_{\cI_{o, Z}, 2D}(1-\eta)=0$ for $0<\eta<<1$.

\subsection{When $Z$ is not fibred by abelian subvarieties}

We first  treat the case $Z$ is not fibred by abelian subvarieties, or in other words, any desingularization of $Z$ is of general type.

By Theorem \ref{subadjunction}, we know that there exists an effective $\mathbb Q$-divisor $D_{Z}$ on $Z$ such that $(Z, D_{Z})$ is a klt pair and $$D_1|_{Z}\sim_{\mathbb Q}K_{Z}+D_{Z}.$$
We then take a log resolution $\mu:\tilde{Z}\rightarrow Z$ and write $$K_{\tilde Z}+\tilde{D}_Z+E_1=\mu^*(K_Z+D_Z)+E_1,$$ where $\tilde{D_Z}$ is the strict transform of $D_Z$, $E_1$, $E_2$ are exceptional divisors, and
$\tilde{D}_Z+E_1+E_2$ has SNC support.

Note that $2\mu^*(D_1)+2E_1\sim_{\mathbb Q}2K_{\tilde Z}+2\tilde{D}_Z+2E_2$. Thus $2\mu^*(D_1)+2E_1-K_{\tilde{Z}}\sim_{\mathbb Q} K_{\tilde Z}+2\tilde{D}_Z+2E_2$. We then take $D_2=2D_1+\epsilon L\sim_{\mathbb
Q}(\frac{c_1}{p+2}+\epsilon)L$ for some $0<\epsilon<<1$ such that $\frac{c_1}{p+2}+\epsilon<\frac{1}{p+2}$. We denote $c_2=\frac{c_1}{p+2}+\epsilon$.
Then \begin{eqnarray}\label{compare-divisor}\mu^*(D_2)+E_2-K_{\tilde Z}\sim_{\mathbb Q} K_{\tilde Z}+D',\end{eqnarray} where $D'$ is a big $\mathbb Q$-divisor on $\tilde{Z}$.

We then take $M>0$ an integer sufficiently large and divisible such that $M^2\epsilon$ and $M^2\frac{c_1}{p+2}$ are integers so that $M^2D_2$ is $\mathbb Q$-equivalent to an integral divisor $L^{\otimes M^2c_2}$. Let
$\pi_M: A\rightarrow A$ be the multiplication-by-$M$ map and we consider the Cartesian
\begin{eqnarray*}
\xymatrix{
\tilde{Z}^{(M)}\ar[d]^{\tilde{\pi}_Z}\ar[r]^{\mu_M} & Z^{(M)} \ar[d]^{\pi_Z}\ar@{^{(}->}[r] &A\ar[d]^{\pi_M}\\
\tilde{Z}\ar[r]^{\mu} & Z\ar@{^{(}->}[r] & A.
}
\end{eqnarray*}
We will also denote by  $\tilde{E}_i=\tilde{\pi}_Z^*E_i$ on $\tilde{Z}^{(M)}$ for $i=1,2$. They are also $\mu_M$-exceptional divisors.

We  claim that $$H^0(\tilde{Z}^{(M)}, \mu_M^*(L^{\otimes M^2c_2}|_{Z^{(M)}})(2\tilde{E_2}-K_{\tilde{Z}^{(M)}}) )\neq 0.$$ Indeed, by (\ref{compare-divisor}), $\mu_M^*(L^{\otimes
M^2c_2}|_{Z^{(M)}})(2\tilde{E_2}-K_{\tilde{Z}^{(M)}})\sim_{\mathbb Q} K_{\tilde{Z}^{(M)}}+\tilde{\pi}_Z^*D'$ and $\tilde{\pi}_Z^*D'$ is a big divisor on $\tilde{Z}^{(M)}$. By \cite[Theorem 11.2.12]{Lar2}, for any nef
divisor $H$ on $A$, $$H^i(\tilde{Z}^{(M)}, \mu_M^*(L^{\otimes M^2c_2}|_{Z^{(M)}})(2\tilde{E_2}-K_{\tilde{Z}^{(M)}}) \otimes \cJ(||\tilde{\pi}_Z^*D'||)\otimes \mu_M^*H)=0$$ for $i>0.$ From this, we deduce that
$$L^{\otimes M^2c_2}|_{Z^{(M)}}\otimes \mu_{M*}\big(\cO_{\tilde{Z}^{(M)}}(2\tilde{E_2}-K_{\tilde{Z}^{(M)}})\otimes \cJ(||\tilde{\pi}_Z^*D'||)\big)$$ is a non-zero IT$^0$ sheaf on $A$ supported on $Z^{(M)}$. Hence it has
global sections and so does $ \mu_M^*(L^{\otimes M^2c_2}|_{Z_M})(2\tilde{E_2}-K_{\tilde{Z}^{(M)}}) $.

We then take a global section $0\neq s\in H^0(\tilde{Z}^{(M)}, \mu_M^*(L^{\otimes M^2c_2}|_{Z^{(M)}})(2\tilde{E_2}-K_{\tilde{Z}^{(M)}}) )$ and let $\tilde{D}$ be the corresponding divisor on $\tilde{Z}^{(M)}$.
We then have
\begin{eqnarray*}
0\rightarrow \cO_{\tilde{Z}^{(M)}}(K_{\tilde{Z}^{(M)}})\xrightarrow{\cdot s} \mu_M^*(L^{\otimes M^2c_2}|_{Z^{(M)}})(2\tilde{E_2})\rightarrow \cO_{\tilde{D}}(K_{\tilde{D}})\rightarrow 0,
\end{eqnarray*}
where $K_{\tilde{D}}=(K_{\tilde{Z}^{(M)}}+\tilde{D})|_{\tilde{D}}$ is the canonical bundle of $\tilde{D}$.  Note that $\mu_{M*}\big(\mu_M^*(L^{\otimes M^2c_2}|_{Z^{(M)}})(2\tilde{E_2})\big)=\mu_M^*(L^{\otimes
M^2c_2}|_{Z^{(M)}})$ since $\tilde{E_2}$ on $\tilde{Z}^{(M)}$ is $\mu_M$-exceptional. Since we have $R^1\mu_{M*} \cO_{\tilde{Z}^{(M)}}(K_{\tilde{Z}^{(M)}})=0$ by Grauert-Riemenschneider, we can apply $\mu_{M*}$ to the above
short exact sequence to get
$$ 0\rightarrow\mu_{M*} \cO_{\tilde{Z}^{(M)}}(K_{\tilde{Z}^{(M)}})\rightarrow L^{\otimes M^2c_2}|_{Z^{(M)}}\rightarrow \mu_{M*}\cO_{\tilde{D}}(K_{\tilde{D}})\rightarrow 0.$$

We claim that $\mu_{M*}\cO_{\tilde{D}}(K_{\tilde{D}})$ is M-regular on $A.$ We already know that $\mu_{M*} \cO_{\tilde{Z}^{(M)}}(K_{\tilde{Z}^{(M)}})$ is GV. On the other hand, $\cI_Z=\cJ(D_1)$. Since $\pi_M$ is \'etale, we
have
$\cI_{Z^{(M)}}=\cJ(\pi_M^*D_1)$. Note that $\pi_M^*D_1\equiv M^2D_1$ and $L^{\otimes M^2c_2}-M^2D_1\equiv (M^2c_2-M^2c_1)$ is an $\mathbb Q$-ample divisor. Thus by Nadel vanishing (see \cite[Theorem 9.4.8]{Lar2}),
 $L^{\otimes M^2c_2}\otimes \cI_{Z^{(M)}}$ is IT$^0$, thus so is $L^{\otimes M^2c_2}|_{Z^{(M)}}$. Thus, $V^i( \mu_{M*}\cO_{\tilde{D}}(K_{\tilde{D}}))\subset V^{i+1}(\mu_{M*} \cO_{\tilde{Z}^{(M)}}(K_{\tilde{Z}^{(M)}}))$ for
 all $i>0$. Hence $\mu_{M*}\cO_{\tilde{D}}(K_{\tilde{D}})$ is M-regular.

 Since we are free to replace $D_1$ (resp. $Z$) by $D_1-p$ (resp. $Z-p$) for some $p\in Z$, we may assume that $\mu$ is an isomorphism around $o$ and we thus still denote by $o$ its preimage in $\tilde{Z}.$ We can also
 assume that $\tilde{\pi}_Z^{-1}(o)\cap \tilde{D}=\emptyset$. We then have
\begin{eqnarray*}0\rightarrow \cO_{\tilde{Z}^{(M)}}(K_{\tilde{Z}^{(M)}})\otimes\cI_{\tilde{\pi}_Z^{-1}(o)}\xrightarrow{\cdot s} \mu_M^*(L^{\otimes M^2c_2}|_{Z^{(M)}})(2\tilde{E_2})\otimes\cI_{\tilde{\pi}_Z^{-1}(o)}\\
\rightarrow \cO_{\tilde{D}}(K_{\tilde{D}})\rightarrow 0.\end{eqnarray*} We again pushforward this short exact sequence to $Z^{(M)}$ and get
\begin{eqnarray}\label{final-sequence}0\rightarrow \mu_{M*}( \cO_{\tilde{Z}^{(M)}}(K_{\tilde{Z}^{(M)}}))\otimes\cI_{\pi_Z^{-1}(o)}\rightarrow L^{\otimes M^2c_2}|_{Z^{(M)}}\otimes\cI_{\pi_Z^{-1}(o)}\nonumber \\
\rightarrow \mu_{M*} \cO_{\tilde{D}}(K_{\tilde{D}})\rightarrow 0.\end{eqnarray}

 Note that $\mu_{M*}( \cO_{\tilde{Z}^{(M)}}(K_{\tilde{Z}^{(M)}}))\otimes\cI_{\pi_Z^{-1}(o)}=\mu_{M*}\big(\tilde{\pi}_Z^*(\cO_{\tilde{Z}}(K_{\tilde{Z}})\otimes \cI_o)\big)$, by \'etale base change, $$\mu_{M*}(
 \cO_{\tilde{Z}^{(M)}}(K_{\tilde{Z}^{(M)}}))\otimes\cI_{\pi_Z^{-1}(o)}=\pi_Z^*(\mu_*(\cO_{\tilde{Z}}(K_{\tilde{Z}})\otimes \cI_o).$$
 By Lemma \ref{ideal-GV}, $\mu_*(\cO_{\tilde{Z}}(K_{\tilde{Z}})\otimes \cI_o$ is a GV sheaf on $A$. Hence $$\mu_{M*}( \cO_{\tilde{Z}^{(M)}}(K_{\tilde{Z}^{(M)}}))\otimes\cI_{\pi_Z^{-1}(o)}$$ is also a GV sheaf on $A.$
 From the short exact sequence (\ref{final-sequence}), we see that $L^{\otimes M^2c_2}|_{Z^{(M)}}\otimes\cI_{\pi_Z^{-1}(o)}$ is again GV. Hence by Subsection \ref{facts} (4), we know that
 $$h^1_{L^{\otimes M^2c_2}|_{Z^{(M)}}\otimes\cI_{\pi_Z^{-1}(o)}, L}(t)=0$$ for $t>0$. Note that $\cI_{\pi_Z^{-1}(o)}=\pi_Z^{-1}\cI_{o, Z}$, thus we have $$h^1_{\cI_{o, Z}, L}(t)=0$$ for $t>c_2$. Because
 $c_2=2c_1+\epsilon<\frac{1}{p+2}$. We finish the proof by Lemma \ref{criterion} when $\tilde{Z}$ is of general type,

\subsection{The general case}
We shall apply induction on dimensions. We now assume that $\beta(L|_B)<\frac{1}{p+2}$ for any proper abelian subvariety $B$ of $A$.

We assume that $Z$ is fibred by an abelian subvariety $K$ of  $A$ and $\overline{Z}\hookrightarrow A/K$ is not fibred by any abelian subvarieties of $A/K$.  We have
\begin{eqnarray*}\xymatrix{
Z\ar[d]^{p_Z}\ar@{^{(}->}[r] &A\ar[d]^p\\
\overline{Z}\ar@{^{(}->}[r] & A/K}
\end{eqnarray*}
 Note that $p_Z$ is smooth, $\overline{o}=p(o)$ is a smooth point of $\overline{Z}$. Then $K $ is exactly the fiber of $p_Z$ over $\overline{o}$. Then it suffices to show that
 \begin{eqnarray}\label{reduction}h^1_{p_Z^{*}(\cI_{\overline{o}, \overline{Z}}), 2D}(1-\epsilon)=0
 \end{eqnarray} for some $0<\epsilon<<1$, because we have $$0\rightarrow p_Z^{*}(\cI_{\overline{o}, \overline{Z}})\rightarrow \cI_{o, Z}\rightarrow \cI_{o, K}\rightarrow 0,$$ hence combining (\ref{reduction}) with the
 assumption that $\beta(L|_K)<\frac{1}{p+2}$, we have $h^1_{\cI_{o, Z}, 2D}(1-\epsilon)=0$, which implies Theorem \ref{main}.

 We do not have a natural line bundle on $A/K$. Hence we need to apply Poincar\'e's reducibility theorem (\cite[Theorem 5.3.5]{BL}), there exists an abelian subvariety $K'$, which is complementary to $K$, i.e. $K'+K=A$ and
 $K'\cap K=G$ is an abelian finite group, such that for the natural addition map $\mu: K'\times K\rightarrow A$ we have $$\mu^*L\simeq (L|_{K'})\boxtimes (L|_{K}).$$ We will write $L_{K'}:=L|_{K'}$ and $L_K:=L|_K$. We also
 note that $\deg \mu=|G|$.

 We also note $\mu^{-1}(Z)=Z'\times K$, where $Z'\hookrightarrow K'$ is indeed isomorphic to the base change $\overline{Z}\times_{A/K}K'$.  Since $\cJ(D_1)=\cI_Z$, we have $\cJ(\mu^*D_1)=\cI_{\mu^{-1}(Z)}$. By \cite[Theorem
 9.5.35]{Lar2}, we know that for $x\in K$ general, $\cJ(\mu^*D_1|_{K'\times\{x\}})=\cI_{\mu^{-1}(Z)}\cdot \cO_{K'\times\{x\}}$. Hence there exists $D_1'\equiv \frac{c_1}{2(p+2)}L_{K'}$ such that $\cJ(D_1')=\cI_{Z'}$ and
 $Z'$ is indeed the unique minimal lc center of $(K', D_1')$.

 We now put everything in one commutative diagram:
\begin{eqnarray}
\xymatrix{
K'\times K \ar[r]^{\mu} & A\\
Z'\times K\ar@{^{(}->}[u] \ar[d]\ar[r]^{\mu_Z} & Z\ar@{^{(}->}[u]\ar[d]^{p_Z}\\
Z' \ar@{_{(}->}[d]\ar[r]^{\tau_Z} & \overline{Z}\ar@{_{(}->}[d]\\
K'\ar[r]^{\tau} & A/K.
}
\end{eqnarray}

Note that $\mu^{*}(p_Z^{*}(\cI_{\overline{o}, \overline{Z}}))=\cI_{\tau_Z^{-1}(\overline{o}), Z'}\boxtimes \cO_K$. We have

$$h^1_{p_Z^{*}(\cI_{\overline{o}, \overline{Z}}), 2D}(t)=\frac{1}{|G|}h^1_{\cI_{\tau_Z^{-1}(\overline{o}), Z'}\boxtimes \cO_K, 2\mu^*D}(t).$$
Hence by K\"unneth formula, in order to prove (\ref{reduction}), we just need to show that $$h^1_{\cI_{\tau_Z^{-1}(\overline{o}), Z'}, \frac{1}{p+2}L_{K'}}(1-\epsilon)=0$$ for $0<\epsilon <<1$.

We now apply exactly the same argument as in the previous case.  We will not go through again the whole argument but rather point out several crucial points. Let $\rho: \tilde{Z}\rightarrow \overline{Z}$ be a
desingularization and as before, we may assume that $\rho$ is an isomorphism over an open neighborhood of $\overline{o}$.
We then consider the Cartesian
\begin{eqnarray*}
\xymatrix{
\tilde{Z}'\ar[d]^{\rho'}\ar[d]\ar[r]^{\tilde{\tau}_Z} & \tilde{Z}\ar[d]^{\rho}\\
Z'\ar[r]^{\tau_Z}& \overline{Z}.}
\end{eqnarray*}

Then since $\overline{Z}$ is of general type, $\rho_*\cO_{\tilde{Z}}(K_{\tilde{Z}})\otimes \cI_{\overline{o}, \overline{Z}}$ is GV by Lemma \ref{ideal-GV}. Hence $$\rho'_*\cO_{\tilde{Z}'}(K_{\tilde{Z}'})\otimes
\cI_{\tau_Z^{-1}(\overline{o}), Z'}=\tau_Z^*(\rho_*\cO_{\tilde{Z}}(K_{\tilde{Z}})\otimes \cI_{\overline{o}, \overline{Z}})$$ is  also GV.

Secondly, since $Z'$ is the minimal lc center of $(K', D_1')$, $D_1'\equiv \frac{c_1}{2(p+2)}L_{K'}$, and $0<c_1<1$, we have $\frac{1}{p+2}L_{K'}|_{Z'}-2\rho'_*(K_{\tilde{Z}'})$ is an effective big $\mathbb Q$-divisor on
$Z'$.

We can then proceed as the previous case.
\end{proof}

 \end{document}